\documentclass[12pt]{article}
\usepackage{amsmath,amssymb,amsthm,amscd}
\usepackage{graphicx}

\title{On $\qit$-factorial terminalizations of nilpotent orbits}
\author{Baohua Fu }

\newtheorem{Thm}{Theorem}[section]
\newtheorem{Lem}[Thm]{Lemma}
\newtheorem{Prop}[Thm]{Proposition}
\newtheorem{Cor}[Thm]{Corollary}
\newtheorem{Conj}{Conjecture}
\newtheorem{Rque}[Thm]{Remark}

\def\cit{{\mathbb C}}

\def\qit{{\mathbb Q}}
\def\zit{{\mathbb Z}}

\def\0{{\mathcal O}}
\def\g{{\mathfrak g}}
\def\h{{\mathfrak h}}
\def\p{{\mathfrak p}}

\def\n{{\mathfrak n}}
\def\X{{\mathcal X}}

\def\mtp{\mathop{\rm mtp}\nolimits}
\def\Ind{\mathop{\rm Ind}\nolimits}
\def\Im{\mathop{\rm Im}\nolimits}
\def\Pic{\mathop{\rm Pic}\nolimits}
\def\Exc{\mathop{\rm Exc}\nolimits}
\def\codim{\mathop{\rm codim}\nolimits}
\def\Kr{\mathop{\rm Kr}\nolimits}
\def\q{{\mathfrak q}}

\def\b{{\mathfrak b}}
\begin{document}
\maketitle
\begin{center}
Abstract
\end{center}
In a recent preprint, Y. Namikawa proposed a conjecture on
$\qit$-factorial terminalizations  of nilpotent orbit closures and
he proved his conjecture for classical simple Lie algebras. In
this paper, we prove his conjecture for exceptional simple Lie
algebras. For the birational geometry, contrary to the classical
case, two new types of Mukai flops appear. We also give
classifications of nilpotent orbit closures in  simple Lie
algebras whose normalization has only $\qit$-factorial or terminal singularities. \\

\begin{center}
R\'esum\'e
\end{center}
Dans un papier r\'ecent, Y. Namikawa a propos\'e une conjecture sur les terminalizations
$\qit$-factorielles des cl\^otures d'orbites nilpotentes et il l'a demontr\'e pour les
alg\`ebres de Lie simples classiques. Dans ce papier, on d\'emontre sa conjecture pour les
alg\`ebres de Lie simples exceptionnels. Pour la g\'eom\'etrie birationelle,  contrairement au cas
classique, deux  nouveaux types de flops de Mukai apparaissent. On  obtient aussi des classifications des
cl\^otures d'orbites nilpotentes dans une al\`ebre de Lie simple dont la normalization admet
que de singularit\'es $\qit$-factorielles ou terminales.   \\


\section{Introduction}

Let $\g$ be a complex simple Lie algebra and $G$ its adjoint
group. For a parabolic subgroup $Q \subsetneq G$, we denote by
$\q$ its Lie algebra and $\q=n(\q)+l(\q)$ its Levi decomposition.
For a nilpotent orbit $\0_t$ in $l(\q)$, Lusztig and Spaltenstein
\cite{L-S} showed that $G \cdot (n(\q)+\bar{\0_t})$ is a nilpotent
orbit closure, say $\bar{\0}$, which depends only on the $G$-orbit
of the pair $(l(\q), \0_t)$. The variety $n(\q)+\bar{\0_t}$ is
$Q$-invariant and the surjective map
$$
\pi: G \times^Q(n(\q)+\bar{\0_t}) \to \bar{\0}
$$
is  generically finite and projective, which will be called {\em a
generalized Springer map}. When $\0_t=0$ and $\pi$ is birational,
we call $\pi$ a {\em Springer resolution}.  An {\em induced orbit}
is a nilpotent orbit whose closure is the image of a generalized
Springer map. A nilpotent orbit is called {\em rigid} if it is not
induced.

Recall that for a variety $X$ with rational Gorenstein
singularities, a {\em $\qit$-factorial terminalization} of $X$ is
a birational projective morphism $p: Y \to X$ such that $Y$ has
only $\qit$-factorial terminal singularities and $p^*K_X=K_Y$.
When $Y$ is furthermore smooth, we call $p$ a {\em crepant
resolution}. In \cite{F1}, the author proved that for nilpotent
orbit closures in a semi-simple Lie algebra, crepant resolutions
are Springer resolutions. In a recent preprint \cite{Nam}, Y.
Namikawa proposed the following conjecture on $\qit$-factorial
terminalizations of nilpotent orbit closures.

\begin{Conj}\label{conj}
 Let $\0$ be a nilpotent orbit in a complex simple Lie
algebra $\g$ and $\tilde{\0}$ the normalization of its closure
$\bar{\0}$. Then one of the following holds:

(1) $\tilde{\0}$ is $\qit$-factorial terminal;

(2)  every $\qit$-factorial terminalization of $\tilde{\0}$ is
given by a generalized Springer map.  Furthermore, two such
terminalizations are connected by Mukai flops (cf. \cite{Nam2}, p.
91).
\end{Conj}

In \cite{Nam}, Y. Namikawa proved his conjecture in the case when
$\g$ is classical. In this paper, we shall prove Conjecture
\ref{conj}  for $\g$ exceptional(Theorem \ref{rigid} and Theorem
\ref{thm}), while for the birational geometry, two new types of
Mukai flops appear (Section \ref{Mukai}). Two interesting results
are also obtained: one is the classification of nilpotent orbit
closures with $\qit$-factorial normalization $\tilde{\0}$
(Proposition \ref{Fact}) and the other is the classification of
nilpotent orbit closures with terminal $\tilde{\0}$ (Proposition
\ref{terminal}). \vspace{0.3 cm}

Here is the organization of this paper. After recalling results
from \cite{B-M}, we first give a classification of induced orbits
which are images of birational generalized Springer maps
(Proposition \ref{key}). Using this result, we completely settle
the problem of $\qit$-factoriality of the normalization of a
nilpotent orbit closure in exceptional Lie algebras (Proposition
\ref{Fact}), which shows the surprising result that only in $E_6$,
$\tilde{\0}$ could be non-$\qit$-factorial. We then prove that for
rigid orbits the normalization of its closure is $\qit$-factorial
and terminal (Theorem \ref{rigid}). For induced orbits whose
closure does not admit a crepant resolution, we shall first prove
that except four orbits (which have $\qit$-factorial terminal
normalizations), there exists a generalized Springer map which
gives a $\qit$-factorial terminalization of $\tilde{\0}$. For the
birational geometry,  unlike the classical case proven by Y.
Namikawa, two new types of flops appear here, which we call  Mukai
flops of type $E_{6,I}^I$ and $E_{6, I}^{II}$ (cf. Section
\ref{Mukai}).   Then we apply arguments  in\cite{Nam} to conclude our proof.
 An interesting corollary
is a classification of nilpotent orbits in a simple exceptional
Lie algebra such that $\tilde{\0}$ has only terminal singularities
(Proposition \ref{terminal}). In the appendix, we shall classify
nilpotent orbits in a simple classical Lie algebra whose normalization is
$\qit$-factorial. In \cite{F1}, we treated this problem, but the argument in the proof
of Theorem 2.6 (\cite{F1})  is not correct.
\vspace{0.3 cm}

{\em Acknowledgements:} The author would like to thank Y. Namikawa
for his corrections and helpful correspondences to this paper. I
thank W. de Graaf for his help on computing in {\bf GAP4}. I am
grateful to M. Brion, S. Goodwin, H. Kraft, G. R\"ohrle, E.
Sommers for their helpful correspondences.

I'm very grateful to Travis Schedler for pointing out a mistake in the published version of Proposition \ref{bcd} and for suggestions on Proposition \ref{p.Pic}.
\section{Preliminaries}

\quad \ In this section, we shall recall some results from
\cite{B-M}.
 Let $W$ be the Weyl group of $G$. The Springer
correspondence (\cite{Sp2}) assigns to each irreducible $W$-module
a unique  pair $(\0, \phi)$ consisting of a nilpotent orbit $\0$
in $\g$ and an irreducible representation $\phi$ of the component
group $A(\0):= G_x/(G_x)^\circ$ of $\0$, where  $x$ is any point
in $\0$ and $(G_x)^\circ$ is the identity component of $G_x$.
 The corresponding irreducible $W$-module will be denoted by  $\rho_{(x, \phi)}.$
 This correspondence is not surjective
onto the set of all pairs $(\0, \phi)$. A pair will be called {\em
relevant} if it corresponds to an irreducible $W$-module, then the
Springer correspondence establishes a bijection between
irreducible $W$-modules and relevant pairs in $\g$. For $G$
exceptional, the Springer correspondence has been completely
worked out in \cite{Sp1} for $G_2$, in \cite{Sho} for $F_4$ and in
\cite{A-L} for $E_n (n=6,7,8)$. We will use the tables in
\cite{Car} (Section 13.3).

Consider a parabolic sub-group $Q$ in $G$. Let $L$ be a Levi
sub-group of $Q$ and $T$ a maximal torus in $L$.  The Weyl group
of $L$ is $W(L) := N_L(T)/T$, where $N_L(T)$ is the normalizer of
$T$ in $L$. It is a sub-group of the Weyl group $W$ of $G$. For a
representation $\rho$ of $W(L)$, we denote by
$\Ind_{W(L)}^W(\rho)$ the induced representation of $\rho$ to $W$.

\begin{Prop}[\cite{B-M}, proof of Corollary 3.9]\label{degree}
Let $\pi: G \times^Q (n(\q)+\bar{\0_t}) \to \bar{\0}_x$ be the
generalized Springer map associated to the parabolic sub-group $Q$
and the nilpotent orbit $\0_t$. Then
$$\deg(\pi) = \sum_{\phi} \mtp(\rho_{(x, \phi)},
\Ind_{W(L)}^W(\rho_{(t, 1)})) \deg \phi,$$ where the sum is over
all irreducible representations $\phi$ of $A(\0_x)$ such that
$(\0_x, \phi)$ is a relevant pair, $\mtp(\rho_{(x, \phi)},
\Ind_{W(L)}^W(\rho_{(t, 1)}))$
 is the multiplicity
 of $ \rho_{(x, \phi)} $
in $ \Ind_{W(L)}^W(\rho_{(t, 1)})$ and $\deg \phi$ is the
dimension of
 the irreducible representation $\phi$.
\end{Prop}

The multiplicity $ \mtp(\rho_{(x, \phi)}, \Ind_{W_0}^W(\rho))$ has
been worked out in \cite{Alv}, for any irreducible representation
$\rho$ of any maximal parabolic sub-group  $W_0$ of $W$. Note that
$ \Ind_{W(L)}^W(\rho) = \Ind_{W_0}^W (\Ind_{W(L)}^{W_0}(\rho))$
for any  sub-group $W_0$ of $W$ containing $W(L)$ and
 $\Ind_{W(L)}^{W_0}(\rho) $ can be determined by the Littlewood-Richardson rules
when $W_0$ is classical and by \cite{Alv} when $W_0$ is
exceptional.

By the remark in section 3.8 \cite{B-M}, $\mtp(\rho_{(x, 1)},
\Ind_{W(L)}^W(\rho_{(t, 1)})) = 1$,  which gives the following
useful corollary.
\begin{Cor}\label{A=1}
If $\0$ is an induced orbit with $A(\0)=\{1\}$, then every
generalized Springer map onto $\0$ is birational.
\end{Cor}

Recall that a complex variety $Z$ of dimension $n$ is called {\em
rationally smooth} at a point $z \in Z$ if
\begin{equation*}
H_i(Z, Z\setminus \{z\}; \qit) = \begin{cases} \qit & \text{if
$i=2n$}, \\
 0 & \text{otherwise}.
\end{cases}
\end{equation*}
For a generalized Springer map $\pi: Z:=G \times^Q
(n(\q)+\bar{\0_t}) \to \bar{\0}_x$, an orbit $\0_{x'} \subset
\bar{\0}_x$ is called {\em $\pi$-relevant} if $2 \dim \pi^{-1}(x')
= \dim \0_x - \dim \0_{x'}.$
\begin{Prop}[\cite{B-M}, Proposition 3.6]\label{relevance}
Assume that $Z$ is rationally smooth at points in $\pi^{-1}(x')$.
Then $\0_{x'}$ is $\pi$-relevant if and only if $$\mtp(\rho_{(x',
1)}, \Ind_{W(L)}^W \rho_{(t,1)}) \neq 0.$$
\end{Prop}

When $t=0$, we have $Z \simeq T^*(G/P)$ is smooth, $\pi$ is the
moment map and $\0_x$ is the Richardson orbit associated to $P$.
In this case, $\rho_{(t,1)} = \varepsilon_{W(L)}$ is the sign
representation and we have a geometric interpretation of the
multiplicity.
\begin{Prop}[\cite{B-M}, Corollary 3.5]\label{component}
For the map $\pi: T^*(G/P) \to \bar{\0}_x$, the multiplicity
$\mtp(\rho_{(x', 1)}, \Ind_{W(L)}^W \varepsilon_{W(L)})$ is the
number of irreducible components of $\pi^{-1}(\0_{x'})$ of
dimension $\dim \0_x + (\dim G/P - 1/2 \dim \0_x)$.
\end{Prop}

\section{Birational generalized Springer maps}\label{prop}
Throughout the paper, we will use notations in \cite{McG} (section
5.7) for nilpotent orbits in exceptional Lie algebras and we use
partitions for those in classical Lie algebras. In this section,
we classify nilpotent orbits in a simple exceptional Lie algebra
which is the image of a birational generalized Springer map. More
precisely, we prove the following proposition.
\begin{Prop} \label{key}
Let $\0$ be an induced nilpotent orbit in a simple complex
exceptional Lie algebra. The closure  $\bar{\0}$ is the image of a
birational generalized Springer map if and only if $\0$ is not one
of the following orbits:
 $A_2+A_1, A_4+A_1$ in $E_7$,
$A_4+A_1, A_4+2A_1$ in $E_8$.
\end{Prop}

By Corollary \ref{A=1}, to prove Proposition \ref{key}, we just
need to consider induced orbits with non-trivial $A(\0)$ but
having no Springer resolutions. The classification of
induced/rigid orbits in exceptional Lie algebras can be found for
example in \cite{McG} (section 5.7). We will use the tables
therein to do a case-by-case check. Note that the group $G$
therein is simply-connected, thus $A(x)$ in these tables is
$\pi_1(\0_x)$. One can get $A(\0)$ by just omitting the copies of
$\zit/d \zit,d=2,3$ when it presents. When $A(\0)$ is $S_2$ (resp.
$S_3$), we will denote by $\epsilon$ (resp. $\epsilon_1,
\epsilon_2$) its non-trivial irreducible representations. By the
remark in section 3.8 \cite{B-M}, $\mtp(\rho_{(x, 1)},
\Ind_{W(L)}^W(\rho_{(t, 1)})) = 1$, so we just need to compute the
multiplicity $\mtp(\rho_{(x, \phi)}, \Ind_{W(L)}^W(\rho_{(t,
1)}))$ for non-trivial $\phi$.

\subsection{$F_4$}

\quad \ There are two orbits to be considered: $B_2$ and
$C_3(a_1)$. The orbit $B_2$ is induced from $(C_3, 21^4)$. We have
$\rho_{(t, 1)} = [1^3:-]$ and $\rho_{(x, \epsilon)} =
\phi_{4,8}=\chi_{4,1}.$ By \cite{Alv} (p. 143), we get
$\mtp(\rho_{(x, \epsilon)}, \Ind_{W(C_3A_1)}^W
\Ind_{W(C_3)}^{W(C_3A_1)}(\rho_{(t,1)})) = 0$, thus the degree of
the associated generalized Springer map  is one. The orbit
$C_3(a_1)$ is induced from $(B_3, 2^21^3)$.  We have $\rho_{(t,
1)}=[-:21]$ and $\rho_{(x, \epsilon)} = \phi_{4,7'}=\chi_{4,4}.$
By \cite{Alv} (p. 147), the degree of $\pi$ is one.

\subsection{$E_6$}

\quad \ When $\g=E_6$, every induced orbit either has
$A(\0)=\{1\}$ or  admits a Springer resolution.

\subsection{$E_7$}\label{E7}

\quad \ We have four orbits to be considered: $A_3+A_2, D_5(a_1)$,
$A_2+A_1$ and $A_4+A_1$.

The orbit $A_3+A_2$ is induced from $(D_6, 32^21^5)$. A calculus
shows that the associated generalized Springer map has degree 2.
By a dimension counting, it is also induced from $(D_5+A_1,
[2^21^6] \times [1^2])$. For this induction, one has $\rho_{(t,
1)}=[1:1^4]\times[1^2]$ and $\rho_{(x, \epsilon)} =
\phi_{84,15}=84_a^*.$ By \cite{Alv} (p. 49), one gets $
\mtp(84_a^*, \Ind_{W(D_5A_1)}^W [1:1^4]\times[1^2]) = \mtp(84_a,
\Ind_{W(D_5A_1)}^W [4:1]\times[2]) = 0 $, thus the induced
generalized Springer map is birational. The orbit $D_5(a_1)$ is a
Richardson orbit but its closure has no Springer resolutions
(\cite{Fu}).  By Thm. 5.3 \cite{McG}, it is induced from $(D_6,
3^22^21^2)$. One finds $\rho_{(t, 1)}=[1^2:21^2]$ and $\rho_{(x,
\epsilon)}=\phi_{336,11} = 336_a^*.$ Now by \cite{Alv} (p. 43),
the degree is one.

 The orbit $A_2+A_1$ has a unique induction (by
dimension counting) given by $(E_6, A_1)$. We have $\rho_{(t,
1)}=6_p^*$ and $\rho_{(x, \epsilon)} = \phi_{105, 26} = 105_a$. By
\cite{Alv} (p. 51),  the degree is 2.  The orbit $A_4+A_1$ is a
Richardson orbit with no crepant resolutions (\cite{Fu}), i.e. the
degree given by the induction $(A_2+2A_1, 0)$ is of degree 2. It
has three other inductions, given by $(E_6, A_2+2A_1)$, $(A_6,
2^21^3)$ and $(A_5+A_1, 2^41^2+0)$. One shows that every such
induction gives a generalized Springer map of degree 2.

\subsection{$E_8$}
\quad \ We need to consider the following orbits: $A_3+A_2,
D_5(a_1), D_6(a_2), E_6(a_3)+A_1, E_7(a_5), E_7(a_4),
E_6(a_1)+A_1, E_7(a_3)$, $A_4+A_1$ and  $A_4+2A_1$. The orbit
$A_3+A_2$ is induced from $(D_7, 2^21^{10})$. We have $\rho_{(t,
1)}=[1:1^6]$ and $\rho_{(x, \epsilon)} = \phi_{972,32}=972_x^*$.
By \cite{Alv} (p. 105), we get $\deg =1$. The orbit $D_5(a_1)$ is
induced from $(E_7, A_2+A_1)$ by Thm. 5.3 \cite{McG}. We have
$\rho_{(t, 1)}=120_a^*$ and $\rho_{(x, \epsilon)} =
\phi_{2100,28}=2100_x^*$. By \cite{Alv} (p. 140), we get $\deg
=1$, while the induction from $(E_6, A_1)$ gives a map of degree
2. The orbit $D_6(a_2)$ is induced from $(D_7, 32^41^3)$. We have
$\rho_{(t, 1)}=[-:2^31]$ and $\rho_{(x, \epsilon)} =
\phi_{2688,20}=2688_y$. By \cite{Alv}(p. 106), we get $\deg=1$.
The orbit $E_6(a_3)+A_1$ is induced from $(E_7, 2A_2+A_1)$. We
have $\rho_{(t, 1)}=\phi_{70,18}=70_a$ and $\rho_{(x, \epsilon)} =
\phi_{1134,20}=1134_y$.  By \cite{Alv}(p. 139), we get $\deg=1$.
The orbit $E_7(a_5)$ has $A(\0)=S_3$ and is induced from
$(E_6+A_1, 3A_1+0)$. We have $\rho_{(t, 1)}=\phi_{15,16} \times
[1^2] = 15_q^* \times [1^2]$, $\rho_{(x, \epsilon_1)} =
\phi_{5600,19}=5600_w, \rho_{(x,\epsilon_2)} =
\phi_{448,25}=448_w$. By \cite{Alv} (p. 136), we get $\deg=1$. The
orbit $E_7(a_4)$ is induced from $(E_7, A_3+A_2)$. We have
$\rho_{(t, 1)}=\phi_{378,14}=378_a$ and $\rho_{(x, \epsilon)} =
\phi_{700,16}=700_{xx}$. By \cite{Alv}(p. 139), we get $\deg=1$.
The orbit $E_6(a_1)+A_1$ is induced from $(E_7, A_4+A_1)$. We have
$\rho_{(t, 1)}=\phi_{512,11}=512_a^*$ and $\rho_{(x, \epsilon)}
=\phi_{4096,12}=4096_x$. By \cite{Alv}(p. 141), we get $\deg=1$.
The orbit $E_7(a_3)$ is induced from $(D_6, 3^22^21^2)$. We have
$\rho_{(t, 1)}=[1^2:21^2]$ and $\rho_{(x, \epsilon)} =
\phi_{1296,13}=1296_z$. By \cite{Alv}(p. 43), we get $
\Ind_{W(D_6)}^{W(E_7)} [1^2:21^2] =
189_b+189_c+315_a+280_a+336_a+216_a+512_a+378_a+420_a. $ Now by
\cite{Alv}(p.138, p.140), we get $\deg=1$.

The orbit $A_4+A_1$ has a unique induction given by $(E_6+A_1,
A_1+0)$, which gives a generalized Springer map of degree 2. The
orbit $A_4+2A_1$ has a unique induction, given by $(D_7, 2^41^6)$.
This gives a map of degree 2.

This concludes the proof of Proposition \ref{key}.

\section{$\qit$-factoriality}\label{qf}

In this section, we study the problem of $\qit$-factoriality of
the normalization of a nilpotent orbit closure.
\begin{Lem} \label{Picard}
Let $\0_x$ be a nilpotent orbit in a complex simple Lie algebra
and $(G_x)^\circ$ the identity component of the stabilizer $G_x$
in $G$. Assume that the character group $\chi((G_x)^\circ)$ is
finite, then $\Pic(\0_x)$ is finite and $\bar{\0}_x$ is
$\qit$-factorial.
\end{Lem}
\begin{proof}
The exact sequence $1 \to (G_x)^\circ \xrightarrow{i} G_x \to
A(\0_x):=G_x/(G_x)^\circ \to 1$ induces an exact sequence: $1 \to
\chi(A(\0_x)) \to \chi(G_x) \to \Im(i^*) \to 1$. By assumption,
$\chi((G_x)^\circ)$ is finite, so is $\Im(i^*)$. On the other
hand, $A(\0_x)$ is a finite group, thus $\chi(A(\0_x))$ is also
finite. This gives the finiteness of $\chi(G_x)$. The exact
sequence $1 \to G_x \to G \xrightarrow{q} \0_x \to 1$ induces an
exact sequence $1 \to \chi(G_x) \to \Pic(\0_x) \to \Im(q^*) \to
1$. As $\Pic(G)$ is finite, so is $\Im(q^*)$. This proves that
$\Pic(\0_x)$ is finite. The last claim follows  from $\codim
(\bar{\0}_x \setminus \0_x) \geq 2$.
\end{proof}
\begin{Rque}
It is a subtle problem to work out explicitly the group
$\Pic(\0_x)$, since in general $q^*, i^*$ are not surjective.
\end{Rque}
\begin{Lem}\label{Q-fac}
Let $\pi: T^*(G/P) \to \bar{\0}$ be a Springer resolution. Then
$\tilde{\0}$ is $\qit$-factorial if and only if the number of
irreducible exceptional divisors of $\pi$ equals to $b_2(G/P)$.
\end{Lem}
\begin{proof}
As $\tilde{\0}$ admits a positive weighted $\cit^*$-action with a
unique fixed point, $\Pic(\tilde{\0})$ is trivial. As a
consequence, $\tilde{\0}$ is $\qit$-factorial if and only if
$\Pic(\0)$ is finite. Let $E_i, i=1, \cdots, k$ be the irreducible
exceptional divisors of $\pi$. We have the following exact
sequence:
$$
 \oplus_{i=1}^k \qit [E_i] \to \Pic(T^*(G/P))\otimes \qit \to
\Pic(\0) \otimes \qit \to 0.
$$ By \cite{Nam} (Lemma 1.1.1), the first map is injective.
Now it is clear that $\Pic(\0)$ is finite if and only if
$k=b_2(G/P)$.

\end{proof}

\begin{Prop}\label{Fact}
Let $\0$ be a nilpotent orbit in a simple exceptional Lie algebra
and $\tilde{\0}$ the normalization of its closure $\bar{\0}$. Then
$\tilde{\0}$ is $\qit$-factorial if and only if
 $\0$ is not one of the following orbits in $E_6$: $2A_1, A_2+A_1,
A_2+ 2A_1, A_3, A_3+A_1, A_4, A_4+A_1, D_5(a_1), D_5$.
\end{Prop}
\begin{proof}
By Lemma \ref{Picard}, we just need to check orbits whose type of
$C$ contains a factor of $T_i$ in the tables of \cite{Car}(Chap.
13, p.401-407). This gives that nilpotent orbit closures in $G_2$
and $F_4$ are $\qit$-factorial.

In $E_6$, there are in total ten orbits  to be considered. The
orbit closures of $2A_1, A_2+2A_1$ have small resolutions by
\cite{Nam2}, thus $\tilde{\0}$ is not $\qit$-factorial. As we will
see in section \ref{Mukai}, the orbit closures of $A_2+A_1,
A_3+A_1$ have small $\qit$-factorial terminalizations, thus
$\tilde{\0}$ is not $\qit$-factorial. The six left orbit closures
have Springer resolutions. We will now use Proposition
\ref{component} to calculate the numbers of irreducible
exceptional divisors and then apply Lemma \ref{Q-fac}. When
$\0=A_3$, a Springer resolution is given by the induction $(A_4,
0)$. The boundary $\bar{\0} \setminus \0 = \bar{\0}_{A_2+2A_1}$
has codimension 2 and $\rho_{(A_2+2A_1,1)} = \phi_{60,11} =
60_p^*$. By \cite{Alv} (p. 31), we get $\mtp = 1$ while
$b_2(G/P)=2$, thus $\tilde{\0}$ is not $\qit$-factorial. When
$\0=D_4(a_1)$, it is an even orbit and a Springer resolution is
given by the induction $(2A_2+A_1, 0)$ The boundary  $\bar{\0}
\setminus \0 = \bar{\0}_{A_3+A_1}$ has codimension 2. By
\cite{Alv} (p.33), we get $\mtp =1 = b_2(G/P)$. This implies that
$\tilde{\0}$ is $\qit$-factorial. For $\0=A_4$, a Springer
resolution is given by the induction $(A_3, 0)$ and $\bar{\0}
\setminus \0 = \bar{\0}_{D_4(a_1)}$ has codimension 2. We find
that $\mtp = 2$ while $b_2(G/P)=3$, thus $\tilde{\0}$ is not
$\qit$-factorial. For $\0=A_4+A_1$, a Springer resolution is given
by the induction $(A_2+2A_1, 0)$ and $\bar{\0} \setminus \0 =
\bar{\0}_{A_4}$ has codimension 2. By \cite{Alv}, we find $\mtp=1$
while $b_2(G/P)=2$, thus $\tilde{\0}$ is not $\qit$-factorial. For
$\0= D_5(a_1)$, a Springer resolution is given by the induction
$(A_2+A_1, 0)$ and $\bar{\0} \setminus \0 = \bar{\0}_{A_4+A_1}
\cup \bar{\0}_{D_4}$. Only $\bar{\0}_{A_4+A_1}$ has codimension 2.
By \cite{Alv}, we find $\mtp=2$ while $b_2(G/P)=3$, thus
$\tilde{\0}$ is not $\qit$-factorial. For $\0=D_5$, a Springer
resolution is given by the induction $(2A_1, 0)$ and $\bar{\0}
\setminus \0 = \bar{\0}_{E_6(a_3)}$ has codimension 2. By
\cite{Alv}, we find $\mtp=1$ while $b_2(G/P)=4$, thus $\tilde{\0}$
is not $\qit$-factorial.

In $E_7$, there are six orbits to be considered. We first consider
the two orbits: $A_4$ and $E_6(a_1)$.  For $\0=A_4$, a Springer
resolution is given by the induction $(A_1+D_4, 0)$ and $\bar{\0}
\setminus \0=\bar{\0}_{A_3+A_2}$ is of codimension 2. Using
\cite{Alv}, we find $\mtp = 2= b_2(G/P)$, thus $\tilde{\0}$ is
$\qit$-factorial. For $\0=E_6(a_1)$, a Springer resolution is
given by the induction $(4A_1, 0)$ and $\bar{\0} \setminus
\0=\bar{\0}_{E_7(a_4)}$ is of codimension 2. Using \cite{Alv}, we
find $\mtp = 3= b_2(G/P)$, thus $\tilde{\0}$ is $\qit$-factorial.

In $E_8$, there are seven orbits to be considered. We first
consider the three orbits: $D_5+A_2, D_7(a_2)$ and  $D_7(a_1)$.
For $\0=D_5+A_2$, a Springer resolution is given by the induction
$(A_2+A_4, 0)$ and
 and $\bar{\0} \setminus \0=\bar{\0}_{E_7(a_4)} \cup \bar{\0}_{A_6+A_1}$ is
of codimension 2.  As both orbits are special,  they are relevant,
so we get $\mtp = 2= b_2(G/P)$, thus $\tilde{\0}$ is
$\qit$-factorial. For $\0=D_7(a_2)$, a Springer resolution is
given by the induction $(2A_3, 0)$ and $\bar{\0} \setminus
\0=\bar{\0}_{D_5+A_2}$ is of codimension 2. Using \cite{Alv}, we
find $\mtp = 2= b_2(G/P)$, thus $\tilde{\0}$ is $\qit$-factorial.
For $\0=D_7(a_1)$,  a Springer resolution is given by the
induction $(A_2+A_3, 0)$ and $\bar{\0} \setminus
\0=\bar{\0}_{E_7(a_3)} \cup \bar{\0}_{E_8(b_6)}$ is of pure
codimension 2. Using \cite{Alv}, we find $\mtp = 3= b_2(G/P)$,
thus $\tilde{\0}$ is $\qit$-factorial.

Now we consider the following orbits:  $A_3+A_2, D_5(a_1)$ in
$E_7$ and $A_3+A_2$ in $E_8$. By the proof of Proposition
\ref{key}, $\bar{\0}$ admits a $\qit$-factorial terminalization
given by a generalized Springer map $\pi: Z:=G \times^P(n(\p) +
\bar{\0}_t) \to \bar{\0}$ with $b_2(G/P)=1$ and $\Pic(\0_t)\otimes
\qit = 0$. One checks easily that for such $\0$, $\bar{\0}
\setminus \0$ contains a unique codimension 2 orbit $\0_{x'}$. We
then use \cite{Alv} to check that $\mtp(\rho_{(x', 1)},
\Ind_{W(L)} \rho_{(t,1)}) \neq 0$. As the variety $Z$ is smooth
along $G \times^P(n(\p)+\0_t)$, one checks that $Z$ is smooth in
codimension 3. We can now apply Prop. \ref{relevance} to deduce
that the pre-image of $\0_{x'}$ under the generalized Springer map
is of codimension 1, thus the map is divisorial. As $b_2(Z)=1$,
this implies that $\tilde{\0}$ is $\qit$-factorial.

Now we consider the orbit: $A_2+A_1$ in $E_7$. By the proof of
Proposition \ref{key}, the induction $(E_7, A_2+A_1)$ of
$\0:=\0_{D_5(a_1)}$ in $E_8$ gives a birational map $Z:=G
\times^P(n(\p)+\bar{\0}_{A_2+A_1}) \xrightarrow{\pi} \bar{\0}.$ We
have $\bar{\0}\setminus \0 = \bar{\0}_{D_4+A_1} \cup
\bar{\0}_{A_4+A_1}$. Only the component $\bar{\0}_{A_4+A_1}$ is of
codimension 2 and one shows that $\pi$ is smooth over points in
$\0_{A_4+A_1}$. By applying the proof of Proposition
\ref{component}, we can show that the number of irreducible
exceptional divisors of $\pi$ is equal to the multiplicity
$\mtp(\rho_{(A_4+A_1, 1)}, \Ind_{W(E_7)}^{W(E_8)} \rho_{(A_2+A_1,
1)})$, which is 1 by \cite{Alv}.  We still denote by $\pi$ the map induced by the normalizations:
$\tilde{Z} \to \tilde{\0}$.
Let $E$ be the $\pi$-exceptional divisor.
Since $\pi$ is projective, one can take a
$\pi$-ample divisor $H$. As $\tilde{\0}$ is $\qit$-factorial, $\pi_*H$
is $\qit$-Cartier, so its pull-back $\pi^*(\pi_*H)$ is defined. One
can write $\pi^*(\pi_*H) - H = aE$ for some rational
number $a$. Since $\pi^*(\pi_*H)$ and $H$ are both $\qit$-Cartier,
$E$ is $\qit$-Cartier.  Take a Weil divisor $D$ of $\tilde{Z}$.
Then one can write $\pi^*(\pi_*D) - D = b E$ for some rational number $b$.
Since $\pi^*(\pi_*D)$ and $bE$ are both $\qit$-Cartier, $D$ is $\qit$-Cartier. This implies that
$\tilde{Z}$ is $\qit$-factorial. Thus
$\tilde{\0}_{A_2+A_1}$ is also $\qit$-factorial.

The claim for the remaining four orbits ($A_4+A_1$ in $E_7$,
$A_4+A_1, A_4+2A_1, E_6(a_1)+A_1$ in $E_8$) is proved by the
following Lemma.
\end{proof}

 For a nilpotent element $x \in \g$, the Jacobson-Morozov theorem gives
an $\mathfrak{sl}_2$-triplet $(x, y, h)$, i.e. $[h, x] = 2 x, [h,
y] = -2y, [x, y] = h$. This triplet makes $\g$ an
$\mathfrak{sl}_2$-module, so we have a decomposition $\g =
\oplus_{i \in \zit} \g_i$, where
 $\g_i = \{z \in \g \mid [h, z] = i z \}.$ The Jacobson-Morozov
parabolic sub-algebra of this triplet is $\p: = \oplus_{i \geq 0}
\g_i$, which can be assumed standard.  Let $P$ be the parabolic
subgroup of $G$ determined by $\p$, then its marked Dynkin diagram
is given by marking the non-zero nodes in the weighted Dynkin
diagram of $x$. The {\em Jacobson-Morozov resolution} of
$\bar{\0}_x$ is given by $\mu: Z:=G \times^P \n_2 \to \bar{\0}_x$,
where $\n_2 := \oplus_{i \geq 2} \g_i$ is a nilpotent ideal of
$\p$.
\begin{Lem}
Let $\0$ be one of the following orbits: $A_4+A_1$ in $E_7$,
$A_4+A_1, A_4+2A_1, E_6(a_1)+A_1$ in $E_8$. Then $\tilde{\0}$ is
$\qit$-factorial.
\end{Lem}
\begin{proof}
We will consider the Jacobson-Morozov resolution $\mu: G \times^P
\n_2  \to \bar{\0}$. By \cite{Nam} (Lemma 1.1.1), $\tilde{\0}$ is
$\qit$-factorial if the number of $\mu$-exceptional irreducible
divisors is equal to $b_2(G/P)$.  To find $\mu$-exceptional
divisors, we will use the computer algebra system {\bf GAP4} to
compute the dimension of the orbit $P \cdot z$ for $z \in \n_2$
(which is the same as $\dim [\p,z]$). We denote by $\beta_j$ the
root vector corresponding to the $j$-th positive root of $\g$ as
present in {\bf GAP4} (see \cite{dG} Appendix B).

Consider first the orbit $\0:=\0_{A_4+A_1}$ in $E_7$. Its
Jacobson-Morozov parabolic subgroup $P$ is given by marking the
nodes $\alpha_1, \alpha_4, \alpha_6$(in Bourbaki's ordering). Let
$Q_1$ (resp. $Q_2$) be the parabolic subgroup given by marking the
nodes $\alpha_1,
 \alpha_6$ (resp.  $\alpha_6$). We have $P \subset Q_1 \subset
 Q_2$. Let $Z_i:=G \times^{Q_i} (Q_i \cdot \n_2)$ and
 $\tilde{Z}_i$ its normalization. The Jacobson-Morozov resolution $\mu$
 factorizes through three contractions:
 $$ Z \xrightarrow{\mu_1} \tilde{Z}_1 \xrightarrow{\mu_2} \tilde{Z}_2
 \xrightarrow{\mu_3} \tilde{\0}.$$
We consider the following three elements in $\n_2$:
$x_1:=\beta_{20} + \beta_{21} + \beta_{25}+\beta_{29},$
$x_2:=\beta_{21}+\beta_{25}+\beta_{26}+\beta_{27}+\beta_{28}+\beta_{29}+\beta_{47},$
$x_3:=\beta_{20}+\beta_{21}+\beta_{28}+\beta_{29}+\beta_{30}+\beta_{31}$.
Let $E_i:=G \times^P \overline{P \cdot x_i}$. Using {\bf GAP4}, we
find $\dim (E_i)= \dim(G/P)+ \dim(P \cdot x_i)=103$, thus $E_i$
are irreducible divisors in $Z$. By calculating the dimensions of
$Q_i \cdot x_j$ using {\bf GAP4}, we get that $\mu_1$ contracts
$E_1$ while $\mu_1(E_2)$ and $\mu_1(E_3)$ are again divisors. The
divisor $\mu_1(E_2)$ is contracted by $\mu_2$ while
$\mu_2(\mu_1(E_3))$ is again a divisor, which is contracted by
$\mu_3$. This shows that the three $\mu$-exceptional divisors
$E_i, i=1,2,3$ are distinct, thus $\tilde{\0}$ is
$\qit$-factorial.  Using the program in \cite{dG}, we find
$\mu(E_1)=\bar{\0}_{A_4}$ and
$\mu(E_2)=\mu(E_3)=\bar{\0}_{A_3+A_2+A_1}$.

For the orbit $A_4+A_1$ in $E_8$, its Jacobson-Morozov parabolic
subgroup $P$ is given by marking the nodes $\alpha_1, \alpha_6,
\alpha_8$. Let $Q_1$ (resp. $Q_2$) be the parabolic subgroup given
by marking the nodes $\alpha_1,  \alpha_8$ (resp. $\alpha_8$).As
before, we define $\tilde{Z}_i$ and $\mu_i$. We consider the
following three elements in $\n_2$:
$x_1:=\beta_{42}+\beta_{57}+\beta_{53}+\beta_{43}$,
$x_2:=\beta_{29}+\beta_{45}+\beta_{56}+\beta_{57}+\beta_{58}+\beta_{59}$,
$x_3:=\beta_{57}+\beta_{56}+\beta_{59}+\beta_{54}+\beta_{61}+\beta_{45}+\beta_{58}$.
We define $E_i$ as before and by using {\bf GAP4} we find that
$E_i, i=1,2,3$ are divisors in $Z$. The map $\mu_1$ contracts
$E_1$, the map $\mu_2$ contracts the divisor $\mu_1(E_2)$ and the
map $\mu_3$ contracts the divisor $\mu_2(\mu_1(E_3))$. This shows
that $E_i, i=1,2,3$ are distinct, thus $\tilde{\0}$ is
$\qit$-factorial. We have furthermore
$\mu(E_1)=\mu(E_2)=\bar{\0}_{A_4}$ and
$\mu(E_3)=\bar{\0}_{D_4(a_1)+A_2}$.

For the orbit $A_4+2A_1$ in $E_8$, its Jacobson-Morozov parabolic
subgroup $P$ is given by marking the nodes $\alpha_4, \alpha_8$.
Let $Q_1$ be the parabolic subgroup given by marking the nodes
$\alpha_8$. We define similarly $\mu_i$. We consider the following
elements in $\n_2$:
$x_1:=\beta_{42}+\beta_{57}+\beta_{53}+\beta_{43}+\beta_{61}$,
$x_2:=\beta_{32}+\beta_{42}+\beta_{47}+\beta_{53}+\beta_{57}+\beta_{61}$.
As before, we define $E_i, i=1,2$, which are divisors by
calculating in {\bf GAP4}. The map $\mu_1$ contracts $E_1$ and the
map $\mu_2$ contracts the divisor $\mu_1(E_2)$, thus $E_1 \neq
E_2$ and $\tilde{\0}$ is $\qit$-factorial. We have furthermore
$\mu(E_1) = \bar{\0}_{A_4+A_1}$ and $\mu(E_2)=\bar{\0}_{2A_3}$.

The orbit $\0:=E_6(a_1)+A_1$ is induced from $(E_7, A_4+A_1)$. The
generalized Springer map $Z:=G \times^P(n(\p)+\bar{\0}_{A_4+A_1})
\xrightarrow{\pi} \bar{\0}$ is birational.  We have
$\bar{\0}\setminus \0 = \bar{\0}_{E_6(a_1)} \cup
\bar{\0}_{D_7(a_2)}$. Only the component $\bar{\0}_{D_7(a_2)}$ is
of codimension 2 and one shows that $\pi$ is smooth over points in
$\0_{D_7(a_2)}$. By applying the proof of Proposition
\ref{component}, we can show that the number of irreducible
exceptional divisors of $\pi$ is equal to the multiplicity
$\mtp(\rho_{(A_4+A_1, 1)}, \Ind_{W(E_7)}^{W(E_8)} \rho_{(A_2+A_1,
1)})$, which is 1 by \cite{Alv}. On the other hand, we have just
proved the $\qit$-factoriality of $\tilde{\0}_{A_4+A_1}$, thus
$\Pic(\0_{A_4+A_1})$ is finite. This gives that $b_2(G
\times^P(n(\p)+\0_{A_4+A_1}))=1$ and $\pi$ contains an exceptional
divisor, thus $\tilde{\0}$ is $\qit$-factorial.
\end{proof}

\section{Rigid orbits}

The aim of this section is to prove Conjecture \ref{conj} for
rigid orbits. The classification of rigid orbits can be found for
example in \cite{McG} (Section 5.7).
\begin{Thm}\label{rigid}
Let $\0$ be a rigid nilpotent orbit in a complex simple Lie
algebra $\g$. Then $\tilde{\0}$ is $\qit$-factorial terminal.
\end{Thm}

\begin{proof}
When $\g$ is classical, this is proven in \cite{Nam}. From now on,
we assume that $\g$ is exceptional. The $\qit$-factoriality of
$\tilde{\0}$ is a direct consequence of Proposition \ref{Fact}. As
$\tilde{\0}$ is symplectic,  it has terminal singularities if
$\codim_{\bar{\0}}(\bar{\0}\setminus \0) \geq 4$. Using the tables
in \cite{McG} (section 5.7, 6.4), we calculate the codimension of
$\bar{\0}\setminus \0$ and  it follows that every rigid orbit
satisfies $\codim_{\bar{\0}}(\bar{\0}\setminus \0) \geq 4$ except
the following orbits: $\tilde{A_1}$ in $G_2$, $\tilde{A_2}+A_1$ in
$F_4$, $(A_3+A_1)'$ in $E_7$, $A_3+A_1, A_5+A_1, D_5(a_1)+A_2$ in
$E_8$.

Consider first the orbit $\0:=\tilde{A_1}$ in $G_2$. Its
Jacobson-Morozov parabolic subgroup is given by marking the node
$\alpha_1$(in Bourbaki's ordering). Consider the Jacobson-Morozov
resolution $ Z:= G \times^P \n_2 \xrightarrow{\mu} \bar{\0}$.  By
\cite{F1}, $\bar{\0}$ has no crepant resolution, thus $\mu$ is not
small. As $b_2(G/P)=1$, there exists one unique $\mu$-exceptional
irreducible divisor $E$. The canonical divisor $K_Z$ is then given
by $K_Z=a E$ with $a>0$. This implies that $\tilde{\0}$ has
terminal singularities. This fact is already known in \cite{Kra}
by a different method.

We now consider the three orbits in $E_8$. Let $ Z:=G \times^P
\n_2  \xrightarrow{\mu} \bar{\0}$ be the Jacobson-Morozov
resolution and $p: Z \to G/P$ the natural projection. Let
$\omega_1, \cdots, \omega_8$ be the fundamental weights of $E_8$.
The Picard group $\Pic(G/P)$ is generated by $\omega_i$  s.t.
$\alpha_i$ is a marked node of $P$. The canonical bundle of $Z$ is
given by $K_Z=p^*(K_{G/P} \otimes det(G \times^P \n_2^*))$. Let
$\cup_j E_j$ be the exceptional locus of $\mu$, which is of pure
codimension 1 since $\tilde{\0}$ is $\qit$-factorial. We have
$K_Z=\sum_j a_j E_j$ with $a_j \geq 0$. Note that if we can show
$K_{G/P}^{-1} \otimes det(G \times^P \n_2)$ is ample on $G/P$,
then $a_j>0$ for all $j$(since $p$ does not contract any
$\mu$-exceptional curve), which will prove that $\tilde{\0}$ has
terminal singularities. As $T^*(G/P) \simeq G \times^P (\oplus_{k
\leq -1} \g_k)$, the line bundle $K_{G/P}^{-1} \otimes det(G
\times^P \n_2)$ corresponds to the character $\wedge^{top}
(\oplus_{k \leq -1} \g_k) \otimes \wedge^{top} (\oplus_{k \geq 2}
\g_k) \simeq \wedge^{top} \g_{-1}$ of $P$. An explicit basis of
$\g_{-1}$ and the action of a Cartan subalgebra $\h$ on it can be
computed using {\bf GAP4}. For the orbit $A_3+A_1$, we get that
$K_Z = p^*(-13 \omega_6-3\omega_8)$. For the orbit $A_5+A_1$, we
get $K_Z = p^*(-3\omega_1-7\omega_4-5\omega_8)$. For the orbit
$D_5(a_1)+A_2$, we get $K_Z=p^*(-7\omega_3-6\omega_6-3\omega_8)$.
This proves the claim for these three orbits.

In a similar way, for the orbit $\tilde{A_2}+A_1$ in $F_4$, we
find that $K_Z=p^*(3 \omega_4-2 \omega_2)$ and for the orbit
$(A_3+A_1)'$ in $E_7$, we obtain $K_Z=p^*(5 \omega_1 -3
\omega_4)$, thus the precedent argument does not apply here.
Instead, we will use another approach.  Recall (\cite{Pan}) that
there exists a 2-form $\Omega$ on $Z:= G \times^P \n_2$ which is
defined at a point $(g,x)\in G \times \n_2$ by:
$\Omega_{(g,x)}((u,m), (u',m')) = \kappa ([u, u'], x) + \kappa
(m', u) -\kappa(m, u')$, where $\kappa( \cdot , \cdot)$ is the
Killing form. The tangent space of $Z$ at the point $(g, x)$ is
identified with the quotient $$\g \times \n_2 /\{ (u, [x,u]) | u
\in \oplus_{i \geq 0} \g_i\}.$$ By Lemma 4.3 in \cite{B},  The
kernel of $\Omega_{(g,x)}$ consists of images of elements $(u,
[x,u])$ with $u \in \oplus_{i \geq -1} \g_i$ such that $[x, u] \in
\n_2$. This shows that $\Omega_{(g,x)}$ is non-degenerate if and
only if the set $\Kr_x:=\{u \in \g_{-1}| [x, u] \in \n_2 \} $ is
reduced to $\{0\}$. Let $s:=\wedge^{top} \Omega$, then $K_Z =
div(s)$ and $s((g,x)) \neq 0$ if and only if $\Kr_x = \{0\}$. To
prove our claim, we just need to show that for a generic point $x$
in every $\mu$-exceptional divisor, the section $s$ vanishes at
$x$, i.e. to show that $\Kr_x \neq \{0\}$.

For the orbit $\tilde{A_2}+A_1$ in $F_4$, we consider the two
elements in $\n_2$:  $x_1:=\beta_{11}+\beta_{12}$ and
$x_2:=\beta_{14}+\beta_{15}+\beta_{16}$. Define $E_i:=G \times^P
\overline{P \cdot x_i}, i=1,2$. Using {\bf GAP4}, we find that
$E_1$ and  $E_2$ are of codimension 1 in $Z$. We have $\mu(E_1) =
\bar{\0}_{\tilde{A_2}}$ and $\mu(E_2) =
\bar{\0}_{A_2+\tilde{A_1}}$, which shows that the two divisors are
distinct. As $b_2(G/P)=2$, we get $\Exc(\mu)=E_1 \cup E_2$.
Consider the two elements in $\g_{-1}$:  $u_1:=\beta_{28}$ and
$u_2:=\beta_{25}-2 \beta_{28}$. Then we have $[x_1, u_1]=0$ and
$[x_2, u_2]=\beta_{12} \in \n_2$, which proves that $u_1 \in
\Kr_{x_1}$ and $u_2 \in \Kr_{x_2}$.
 From this we get that $K_Z = a_1 E_1 + a_2 E_2$ with $a_i >0, i=1,2.$

For the orbit $(A_3+A_1)'$ in $E_7$, we consider the two elements
in $\n_2$: $x_1:=\beta_{20}+\beta_{21}+\beta_{49}$ and
$x_2:=\beta_{20}+\beta_{34}+\beta_{35}+\beta_{37}+\beta_{43}+\beta_{45}$.
We define in a similar way $E_1, E_2$ which are divisors by a
calculus in {\bf GAP4}. We have $\mu(E_1)=\bar{\0}_{A_3}$ and
$\mu(E_2)=\bar{\0}_{2A_2+A_1}$, thus $\Exc(\mu)=E_1 \cup E_2$.
Consider the two elements in $\g_{-1}$:  $u_1:=\beta_{67}$ and
$u_2:=\beta_{64}- \beta_{79}-\beta_{81}$. Then we have $[x_1,
u_1]=0$ and $[x_2, u_2]=-\beta_{26}-\beta_{27}+\beta_{40} \in
\n_2$, which proves that $u_1 \in \Kr_{x_1}$ and $u_2 \in
\Kr_{x_2}$. We deduce that $K_Z=a_1 E_1+a_2E_2$ with $a_i>0,
i=1,2,$ which concludes the proof.
\end{proof}
\begin{Rque}
The three orbits in $E_8$ can also be dealt with in the same way.
 Thus in this paper, the essential point where we used {\bf GAP4}  is to compute
the dimension of $[\p, x]$ (surely we have used it in a crucial
way to find the elements $x_i$ in $\n_2$ and $u_i$ in $\g_{-1}$).
\end{Rque}

\begin{Cor}
The normalization $\tilde{\0}$ is smooth in codimension 3 for the
following orbits: $\tilde{A_1}$ in $G_2$, $\tilde{A_2}+A_1$ in
$F_4$, $(A_3+A_1)'$ in $E_7$, $A_3+A_1, A_5+A_1, D_5(a_1)+A_2$ in
$E_8$. In particular, the closure $\bar{\0}$ of any of these
orbits is non-normal.
\end{Cor}

Although the complete classification of $\0$ with normal closure
is unknown in $E_7$ and  $E_8$, E. Sommers communicated to the
author that the orbits in the corollary are known to have
non-normal closures.

\section{Induced orbits}

Recall (\cite{F1}, \cite{Fu}) that a  nilpotent orbit closure in a
simple Lie algebra admits a crepant resolution if and only if it
is a Richardson orbit but not in the following list: $A_4+A_1,
D_5(a_1)$ in $E_7$, $E_6(a_1)+A_1, E_7(a_3)$ in $E_8$.   On the
other hand, by \cite{Nam3}, if $\bar{\0}$ admits a crepant
resolution, then any $\qit$-factorial terminalizations of
$\bar{\0}$ is in fact a crepant resolution. Furthermore the
birational geometry between their crepant resolutions are
well-understood (\cite{Nam2}, \cite{Fu}). Thus to prove Conjecture
\ref{conj}, we will only consider induced orbits whose closure
does not admit any crepant resolution.

\begin{Thm}\label{thm}
Let $\0$ be an induced nilpotent orbit in a complex simple
exceptional Lie algebra $\g$. Assume that $\bar{\0}$ admits no
crepant resolution.  Then

(i) The variety $\tilde{\0}$ has $\qit$-factorial terminal
singularities for the following induced orbits: $A_2+A_1, A_4+A_1$
in $E_7$ and $A_4+A_1, A_4+2A_1$ in $E_8$.

(ii) If $\0$ is not in the list of (i), then any $\qit$-factorial
terminalization of $\tilde{\0}$ is given by a generalized Springer
map. Two $\qit$-factorial terminalizations of $\tilde{\0}$ are
connected by Mukai flops of type $E_{6,I}^I$ or $E_{6, I}^{II}$
(defined in section \ref{Mukai}).
\end{Thm}
\begin{Rque}
Unlike the classical case proved in \cite{Nam}, for an orbit $\0$
such that $\bar{\0}$ has no Springer resolution, the Mukai flops
of type $A-D-E_6$ defined in \cite{Nam2} (p. 91) do not appear
here (see Lemma \ref{bir}).
\end{Rque}

\subsection{Mukai flops}\label{Mukai}
Let $P$ be one of the maximal parabolics in $G:=E_6$ corresponding
to the following marked Dynkin diagrams:

\includegraphics*[width=12cm]{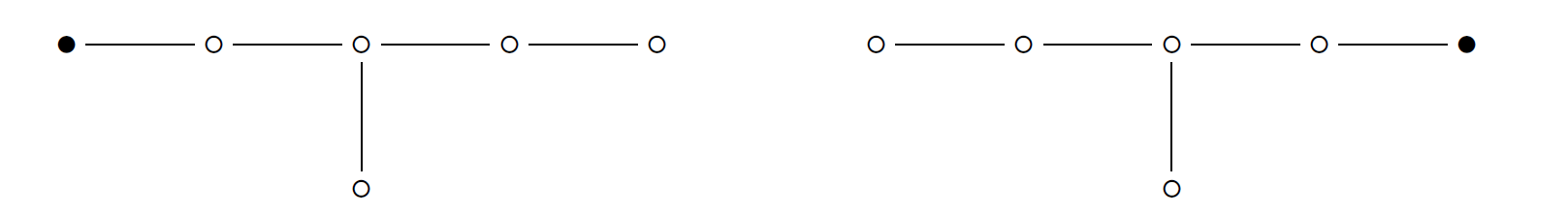}
%

 The Levi part of $P$
is isomorphic to $D_5$. We denote by $\0_I$ (resp. $\0_{II}$) the
nilpotent orbit in $l(\p)$ corresponding to the partition $2^21^6$
(resp. $32^21^3$). Then we have two generalized Springer maps
$\pi_I, \pi_{II}$ with image being the closures of orbits
$A_2+A_1, A_3+A_1$ respectively. As the component group
$A(\0_{A_2+A_1})=A(\0_{A_3+A_1})=\{1\}$, both maps are birational.
By \cite{Nam}, $\tilde{\0}_I, \tilde{\0}_{II}$ are
$\qit$-factorial terminal, thus $\pi_I, \pi_{II}$ give
$\qit$-factorial terminalizations.
\begin{Lem}
The two maps $\pi_I, \pi_{II}$ are small, i.e. the exceptional
locus has codimension at least 2.
\end{Lem}
\begin{proof}
 For $\pi_I$, we have $codim(\bar{\0}_{A_2+A_1}\setminus \0_{A_2+A_1})=4.$ As $\pi_I$
 is semi-small, this implies the claim. For $\pi_{II}$, the orbit
 closure $\bar{\0}_{A_3+A_1}$ is equal to $\0_{A_3+A_1} \cup \bar{\0}_{A_3} \cup
 \bar{\0}_{2A_2+A_1}$.  The codimension of $\bar{\0}_{A_3}$ in
 $\bar{\0}$ is 4, so its pre-image has codimension at least 2. The
 codimension of $\bar{\0}_{2A_2+A_1}$ in $\bar{\0}$ is 2. As one sees easily,
 $G \times^P (n(\p)+\bar{\0}_t)$ is smooth over points in  $\0_{2A_2+A_1}$.  By
 Proposition \ref{relevance}, we need to check
 $\mtp(\rho_{(2A_2+A_1, 1)}, \Ind_{W(D_5)}^W \rho_{(\0_{II},1)}) = 0.$
By \cite{Car}, we have $ \rho_{(2A_2+A_1, 1)}=10_s$ and
$\rho_{(\0_{II},1)} = [-:2^21]$. By \cite{Alv} (p. 31), we get the
claim.
\end{proof}

When $P$ changes from one parabolic to the other, we get two
$\qit$-factorial terminations of the same orbit. The birational
map between them is then a flop, which we will call {\em Mukai
flop of type $E_{6,I}^I$ and $E_{6,I}^{II}$} respectively.

\subsection{Proof of the theorem}
For an orbit $\0$ in  list (i) of the theorem, the variety
$\tilde{\0}$ is $\qit$-factorial by Proposition \ref{Fact}. One
checks using tables in Section 5.7  and 6.4 \cite{McG} that
$\codim(\bar{\0} \setminus \0) \geq 4$, thus $\tilde{\0}$ has only
terminal singularities. This proves claim (i) in the theorem.

 Let now $\0$ be an induced
orbit not in list (i). By Proposition \ref{key} we have a
birational generalized Springer map
$$\pi: G \times^Q(n(\q)+\bar{\0_t}) \to \bar{\0}.$$
For orbits listed in the proof of Proposition \ref{key}, we check
from the above and from Theorem \ref{rigid} that for our choice of
$\0_t$, the variety $\tilde{\0_t}$ is either $\qit$-factorial
terminal or it admits a $\qit$-factorial terminalization given by
a generalized Springer map. For orbits with $A(\0)=\{1\}$, it is
induced from a rigid orbit $\0_t$ and $\tilde{\0}_t$ is
$\qit$-factorial terminal by Theorem \ref{rigid}. This shows that
$\tilde{\0}$ admits a $\qit$-factorial terminalization given by a
generalized Springer map.

 The following
proposition is analogous to Proposition (2.2.1) in \cite{Nam}.
\begin{Prop}
Let $\0$ be a nilpotent orbit in a simple exceptional Lie algebra
such that $\bar{\0}$ does not admit a Springer resolution. Suppose
that a $\qit$-factorial terminalization of $\tilde{\0}$ is given
by the normalization of $G \times^Q(n(\q)+\bar{\0_t})$ for some
parabolic $Q$ and some nilpotent orbit $\0_t$ in $l(\q)$. Assume
that $b_2(G/Q)=1$ and the $\qit$-factorial terminalization is
small. Then this generalized Springer map is one of those in
Section \ref{Mukai}.
\end{Prop}
\begin{proof}
Assume that $\0$ is neither the orbit $A_2+A_1$ nor $A_3+A_1$ in
$E_6$. As we only consider $\0$ such that $\bar{\0}$ has no
Springer resolutions, by Proposition \ref{Fact}, $\tilde{\0}$ is
$\qit$-factorial. This implies that every $\qit$-factorial
terminalization of $\tilde{\0}$ is divisorial, which concludes the
proof.
\end{proof}

Now one can argue as in the proof of Theorem (2.2.2) in \cite{Nam}
to show that every $\qit$-factorial terminalization of $\0$ not in
(i) is actually given by a generalized Springer map and two such terminalizations
are connected by Mukai flops.  Then the following lemma  concludes the proof of the theorem.
\begin{Lem}\label{bir}
For an orbit $\0$ in the theorem but not in  the list (i), any two
$\qit$-factorial terminalizations of $\tilde{\0}$ given by
generalized Springer maps are connected by Mukai flops of type
$E_{6,I}^I$ or $E_{6, I}^{II}$.
\end{Lem}
\begin{proof}
Consider a $\qit$-factorial terminalization  given by the
generalized Springer map associated to $(Q, \0_t)$ with $\0_t \neq
0$. Note that if $l(\q)$ is of type $A$, then $\bar{\0}$ admits a
Springer resolution, which contradicts our assumption. This allows
us to consider only the following situations (for the other cases,
there exists a unique conjugacy class of parabolic subgroups with
Levi part being $l(\q)$): i) $l(\q)$ is $D_5$ in $E_n, n=6,7,8$.
ii) $l(\q)$ is either $D_4+A_1$ or $D_5+A_1$ in $E_8$.

Consider case i). In $E_6$, this is given by the definition of
Mukai flops. In $E_7$, the induction $(D_5, 32^21^3)$ gives two
$\qit$-factorial terminalization of the orbit closure
$\bar{\0}_{D_6(a_2)}$, which are connected by a Mukai flop of type
$E_{6, I}^{II}$. The induction $(D_5, 2^21^6)$ gives the even
orbit $A_4$. In $E_8$, the induction $(D_5, 32^21^3)$ gives two
$\qit$-factorial terminalization of the orbit closure
$\bar{\0}_{E_7(a_2)}$, which are connected by a Mukai flop of type
$E_{6, I}^{II}$, while the induction $(D_5, 2^21^6)$ gives the
even orbit $E_6(a_1)$.

Consider case ii). The induction $(D_4+A_1, 32^21+1^2)$ (resp.
$(D_4+A_1, 2^21^4+1^2)$) gives the even orbit $E_8(b_4)$ (resp.
$E_8(a_6)$). The induction $(D_5+A_1, 32^21^3 + 1^2)$ gives the
even orbit $D_7(a_1)$, while the induction $(D_5+A_1, 2^21^6 +
1^2)$ of $E_7(a_4)$ gives a generalized Springer map of degree 2.
\end{proof}

To conclude this paper, we give the following classification of
nilpotent orbits such that $\tilde{\0}$ has only terminal
singularities.
\begin{Prop}\label{terminal}
Let $\0$ be a nilpotent orbit in a simple complex exceptional Lie
algebra. Then the normalization $\tilde{\0}$ has only terminal
singularities if and only if $\0$ is one of the following orbits:

(1) rigid orbits;

(2) $2A_1, A_2+A_1, A_2+2A_1$ in $E_6$, $A_2+A_1$, $A_4+A_1$ in
$E_7$, $A_4+A_1$, $A_4+2A_1$ in $E_8$.
\end{Prop}
\begin{proof}
By using tables in Section 5.7 and 6.4 of \cite{McG}, we get that
for the three orbits in $E_6$ of (2), we have $\codim(\bar{\0}
\setminus \0) \geq 4$, thus $\tilde{\0}$ has only terminal
singularities. By Theorem \ref{rigid} and Theorem \ref{thm}, this
implies that the variety $\tilde{\0}$ has only terminal
singularities for orbits in (1) and (2).

Assume now $\0$ is not in the list, then by Theorem \ref{thm},
$\tilde{\0}$ admits a $\qit$-factorial terminalization given by a
generalized Springer map. Thus if $\tilde{\0}$ is
$\qit$-factorial, then $\tilde{\0}$ is not terminal. By
Proposition \ref{Fact}, we may assume that $\0$ is one of the
following orbits in $E_6$: $A_3, A_3+A_1, A_4, A_4+A_1, D_5(a_1)$
and $D_5$. As for these orbits except $A_3+A_1$, the closure
$\bar{\0}$ admits a crepant resolution,  thus $\tilde{\0}$ is not
terminal.

Now consider the orbit $\0:=A_3+A_1$ in $E_6$. We will use the
method in the proof of Theorem \ref{rigid} to show that
$\tilde{\0}$ is not terminal. Consider the Jacobson-Morozov
resolution $Z:=G \times^P \n_2 \xrightarrow{\mu} \bar{\0}$, where
$P$ is given by marking the nodes $\alpha_2, \alpha_3, \alpha_5$.
We have $\bar{\0} \setminus \0 = \bar{\0}_{A_3} \cup
\bar{\0}_{2A_2+A_1}$. We consider the following two elements in
$\n_2$: $x_1:=\beta_{17}+\beta_{15}+\beta_{20}$ and
$x_2:=\beta_{17}+\beta_{18}+\beta_{20}+\beta_{21}+\beta_{24}$. We
define $E_i:=G \times^P \overline{P \cdot x_i}, i=1,2,$ which are
irreducible divisors by a calculus in {\bf GAP4}. We have
furthermore $\mu(E_1)=\bar{\0}_{A_3}$ and $\mu(E_2)=
\bar{\0}_{2A_2+A_1}$, thus the two divisors are distinct. As
$b_2(G/P)=3$ and $\tilde{\0}$ is non-$\qit$-factorial, $E_1$ and $
E_2$ are the only two $\mu$-exceptional irreducible divisors.
Using a calculus in {\bf GAP4}, we can show that $\Kr_{x_1}:=\{u
\in \g_{-1} | [x_1, u] \in \n_2\}$ is reduced to $\{0\}$ and
$\Kr_{x_2} \neq \{0\}$. This implies that $K_Z = a E_2$ for some
$a>0$, which proves that $\tilde{\0}$ is not terminal.
\end{proof}
\begin{Rque}
For classical simple Lie algebras, it is proven in \cite{Nam}(Prop. 1.3.2, the proof works  also for type $A$) that
$\tilde{\0}_{d}$ has terminal singularities if and only if the partition $d$ has full members, i.e.
there exists $k$ such that $d=(1^{d_1}2^{d_2}\cdots k^{d_k}$ with $d_i \geq 1$ for all $i=1, \cdots, k$.
\end{Rque}

\section{Appendix: $\qit$-factoriality in classical cases}

In Section \ref{qf}, we have classified nilpotent orbits $\0$ in a simple exceptional Lie algebra $\g$
such that $\tilde{\0}$ is $\qit$-factorial. In the case when $\g$ is of classical
type, it was claimed in \cite{F1} that $\tilde{\0}$ is always $\qit$-factorial, which is not true.
In fact, the proof of Theorem 2.6 in \cite{F1} is not correct.
The aim of this appendix is to complete the classification of
$\0$ such that $\tilde{\0}$ is $\qit$-factorial.

Consider first the case $\g=\mathfrak{sl}_n$ and $\0$ is given by a partition
$d=[d_1, \cdots, d_k]$ of $n$. Let $q_i:=\# \{j|d_j \geq i\}$ be the dual partition. Let $Q \subset SL_n$ be a parabolic subgroup with flag type
$(q_1, \cdots, q_{d_1})$. There exists a Springer map $\mu: T^*(SL_n/Q) \to \bar{\0}$.  Note that
$b_2(G/Q) =d_1-1$.

As the map $\mu$ is semi-small, the codimension one components of $\Exc(\mu)$ come
from those in $\mu^{-1}(\0')$, where $\0' \subset \bar{\0}$ is an orbit of codimension 2.
A codimension 2 orbit $\0'$ in $\bar{\0}$ corresponds to some $i$ such that $d_i - d_{i+1} \geq 2$ and the partition
is given by $d'=[d_1, \cdots, d_{i-1}, d_i-1, d_{i+1}+1, d_{i+2}, \cdots]$. By \cite{KP1}, the singularity of this minimal degeneration is of surface type $A_{d_i-d_{i+1}-1}$. In particular,  the number of
irreducible components of codimension 1 in $\mu^{-1}(\0')$ is at most $d_i-d_{i+1}-1$.  Thus the total number
of irreducible components in $\Exc(\mu)$ is at most $\sum_{i|d_i-d_{i+1} \geq 2} (d_i-d_{i+1}-1) = \sum_{i|d_i-d_{i+1} \geq 1} (d_i-d_{i+1}-1) = d_1 - k$, where $k$ is the number of distinct $d_i$ in $d$, which is strictly smaller
than $b_2(G/Q)$ if $k>1$. By Lemma \ref{Q-fac}, the variety $\tilde{\0}$ is not $\qit$-factorial. Note that when $k=1$, i.e. $d=(a^k)$, we have $Pic(\0) \simeq \zit/a \zit$, which is finite, so
$\tilde{\0}$ is $\qit$-factorial. This proves the following proposition.
\begin{Prop}\label{sl}
For a nilpotent orbit $\0_d$ in $\mathfrak{sl}_n$, the variety $\tilde{\0}_d$ is $\qit$-factorial if and only if
$d=(a^k)$ for some integer $a, k$ such that $ak=n$.
\end{Prop}

For a partition $d$, we put $r_i = \# \{j | d_j = i\}$. Let $G$ be a simply-connected
Lie group with Lie algebra $\g$ and $G^\phi$ the subgroup as defined in \cite{CM}(Section 3.7).
We have $\chi(G_x) = \chi(G^\phi)$ (Lemma 2.2 \cite{F1}), where $G_x$ is the stabilizer of $x \in \0_d$.  Note that
$Pic(\0_d) = \chi(G^\phi)$ since $Pic(G)$ is trivial. Thus if $\chi(G^\phi)$ is finite, then $\tilde{\0}$ is $\qit$-factorial.

For any group $H$, let $H^n_\Delta$ be the diagonal copy of $H$  in the product $H^n$. For $h \in H$, we
denote by $(h)^n_\Delta$ the image of $h$ in $H^n_\Delta$.
If $H_1, \cdots, H_m$ are matrix groups, we denote by $S(\prod_i H_i)$ the subgroup of $\prod_i H_i$
consisting of $m$-tuples of matrices whose determinants have product 1.
Recall the following result(\cite{CM} theorem 6.1.3).
\begin{Prop}[Springer-Steinberg]
\begin{equation*} G^\phi = \begin{cases}
 \prod_{i \, odd} (Sp_{r_i})^i_\Delta \times \prod_{i\,even}(O_{r_i})^i_\Delta   & \g =\mathfrak{sp}_{2n}; \\
 double\ cover \ of\  C:=S(\prod_{i \, even} (Sp_{r_i})^i_\Delta \times \prod_{i\,odd}(O_{r_i})^i_\Delta)  & \g = \mathfrak{so}_{m}.
\end{cases} \end{equation*}
\end{Prop}

The main result of this appendix is as follows, which generalizes Corollary (1.4.3) \cite{Nam}. 
\begin{Prop}\label{bcd}
Let $\g$ be a complex simple Lie algebra of type $BCD$ and $\0$ a nilpotent orbit in $\g$.
The variety $\tilde{\0}$ is $\qit$-factorial except when $\g=\mathfrak{so}_{4n+2}$ and the partition of $\0$ has a unique odd $d_i$ which has multiplicity 2.
\end{Prop}
\begin{proof}
When $\g=\mathfrak{sp}_{2n}$, we have $(G^\phi)^\circ =\prod_{i \, odd} (Sp_{r_i})^i_\Delta \times \prod_{i\,even}(SO_{r_i})^i_\Delta.$ Note that although $SO_2\simeq \cit^*$, the restriction
$\chi (O_2) \to \chi(SO_2)$ is trivial. This implies that $\chi(G^\phi) \to \chi((G^\phi)^\circ)$
 is trivial, which gives $\chi(G^\phi) = \chi(G^\phi/(G^\phi)^\circ)
=\chi(\pi_1(\0))$. Thus $Pic(\0)=\chi(\pi_1(\0))$ and $\tilde{\0}$ is $\qit$-factorial.

Let $\g$ be of type $B-D$. To show the $\qit$-factoriality of $\tilde{\0}$, it suffices
 to show that $\chi(C)$ is finite.
We may assume that there  exits an odd $i$ such that $r_i=2$, otherwise we are done.
 We have two cases:

{\em Case 1) there exist two different odd $i, j$ such that $r_i=2$ and $r_j>0$.}

 In this case, $\chi(C)=\chi(H)$, where $H:=S(\prod_{k\,odd}(O_{r_k})^k_\Delta)$. We have $H^\circ \simeq  \prod_{k\,odd} SO_{r_k}$. It suffices to show that the restriction $\chi(H) \to \chi(H^\circ)$ is trivial. Let $\varphi\in \chi(H)$ be a character.  We define $f: \cit^* \to SO_2$ and $f_{-}: \cit^* \to O_2$ by the following:

$$f(b) = \begin{pmatrix} b & 0 \\0 & b^{-1}
\end{pmatrix},  f_{-}(b) = \begin{pmatrix} o & b \\b^{-1} & 0
\end{pmatrix}.$$

Take any element $g \in O_{r_j} \setminus SO_{r_j}$, then the element $f_{-}^g(b):=(f_{-}(b))^i_\Delta \times (g)^j_\Delta$ is in $H$ for any $b \in \cit^*$. Let $f^g(b):=(f(b))^i_{\Delta} \times (g)^j_\Delta$, then
$f^{g^2}(b)$ is in  $H$.
  Note that $f_{-}^g(1) \cdot f^{1}(b^{-1}) =f_{-}^g(b)$ and $f_{-}^g(b) f_{-}^g(c) = f^{g^2}(bc^{-1})= f^{g^2}(1) \cdot f^1(bc^{-1})$. This implies that $\varphi(f^{1}(bc^{-1})) =\varphi (f^{1}(b^{-1}) f^{1}(c^{-1}))$, which gives that $\varphi(f^{1}(b^2))=1$, for any $b \in \cit^*$. This shows that $\varphi|_{SO(2)} = 1$, and then $\varphi|_{H^0} = 1$.
In particular, this gives that $\tilde{\0}$ is $\qit$-factorial.

{\em Case 2) there exists a unique odd $i$ such that $r_i>0$.}

 In this case, we have $r_i=2$ and $\chi(C) = \chi(SO_2) \simeq \mathbb{Z}$. This implies that $\chi(G^\phi)$ is also isomorphic to $\mathbb{Z}$, hence $\tilde{\0}$ is not $\mathbb{Q}$-factorial.
\end{proof}

Now we can correct Theorem 2.6 of \cite{F1} as follows, whose proof is similar to {\em loc. cit.}, noting that only the case (i) can have infinite character group.
\begin{Prop}\label{p.Pic}
Let $\g$ be a simple complex Lie algebra  of B-C-D type. 

(i) if $\g=\mathfrak{so}_{4n+2}$ and the partition of $\0$ has a unique odd $d_i$ which has multiplicity 2, then 
$Pic(\0) \simeq \mathbb{Z}$;

(ii) We have $Pic(\0) = \X(\pi_1(\0))$ for other cases, where $\pi_1(\0)$ is the fundamental group of $\0$.
\end{Prop}

\end{document}